\documentclass[final,3p,times]{elsarticle}

\usepackage{amsmath,amssymb,amsthm,mathtools,bm}
\usepackage{booktabs,array,longtable,multirow}
\usepackage{graphicx}
\usepackage{placeins}
\usepackage{float}
\usepackage{subcaption}
\usepackage{microtype}
\usepackage{enumitem}
\usepackage{algorithm}
\usepackage{algpseudocode}
\usepackage[colorlinks=true,linkcolor=blue,citecolor=blue,urlcolor=blue]{hyperref}
\usepackage[nameinlink,noabbrev]{cleveref}

\allowdisplaybreaks
\newtheorem{theorem}{Theorem}[section]
\newtheorem{lemma}[theorem]{Lemma}

\theoremstyle{definition}

\theoremstyle{remark}

\crefname{corollary}{corollary}{corollaries}
\Crefname{corollary}{Corollary}{Corollaries}

\newcommand{\dt}{\Delta t}
\newcommand{\Lop}{\mathcal L}
\newcommand{\Lt}{\dot{\mathcal L}}

\newcommand{\one}{\bm 1}
\newcommand{\es}{\bm e_s}

\newcommand{\RR}{\mathbb R}
\newcommand{\PP}{\mathbb P}
\newcommand{\OO}{\mathcal O}

\newcommand{\lc}{\operatorname{lc}}
\newcommand{\coef}{\operatorname{coef}}
\newcommand{\Ttr}{\mathcal T}

\journal{Applied Numerical Mathematics}

\begin{document}
\begin{frontmatter}

\title{Arbitrary-Order Pad\'e-Closed Anchored Two-Derivative Time Discretizations:\\ $s$ Active Stages, Order $2s$, and $L$-Stability}

\author[hpu,bit]{Zhixin Huo\corref{cor1}}
\ead{zhixinhuo@hpu.edu.cn}
\cortext[cor1]{Corresponding author}
\address[hpu]{School of Mathematics and Information Science, Henan Polytechnic University, Jiaozuo, Henan 454003, China}
\address[bit]{School of Mechatronical Engineering, Beijing Institute of Technology, Beijing 100081, China}

\begin{abstract}
An arbitrary-order family of implicit two-derivative one-step methods is
constructed in an anchored active-stage formulation.  At each information
node the method uses both the vector field and its first total time
derivative, enriching the local Hermite data without increasing the number of
unknown stage states.  With the known initial value retained as an anchor and
$s$ unknown active stages, $2s$ Hermite moment conditions yield global order
$2s$.  The two remaining coefficients in each stage row are fixed by the
second-subdiagonal Pad\'{e} approximant $[s-1/s+1]_{e^z}$.  For every ordered
real node set, a Pad\'{e}--Hermite basis theorem proves that the closure is
unique, preserves all moment conditions, and gives
$\det(I-zA-z^2\widehat A)=Q_s(z)$ and
$R_s(z)=P_s(z)/Q_s(z)$.  Hence the coupled stage system has no hidden poles
and the accepted one-step map is $L$-stable (and therefore $A$-stable) for
every positive integer $s$.  Exact symbolic verification is reported through $s=6$, and
high-precision computations confirm orders $2,4,6,$ and $8$ for the first
four members.  At equal active-stage count, comparisons with Gauss--Legendre
and Radau IIA methods demonstrate the combined high-order accuracy and strong
stiff damping of the construction over broad step-size ranges.
\end{abstract}


\end{frontmatter}

\section{Introduction}
\label{sec:introduction}

High-order time integration can be obtained either by introducing more
solution stages or by enriching the information carried by each stage.
Two-derivative methods follow the second strategy: besides the vector field
$\Lop(u)$, they use its first total time derivative
$\Lt(u)=d\Lop(u)/dt$.  Classical multiple-node and multiderivative formulas
were studied in \cite{KastlungerWanner1972,HairerWanner1973}; modern work
includes explicit, implicit, strong-stability-preserving, and IMEX
constructions
\cite{ChanTsai2010,SealGucluChristlieb2014,ChristliebEtAl2016,GottliebGrantHuShu2022,KalogiratouEtAl2023}.
Second-derivative general linear methods have been studied from the viewpoints
of maximal order, Runge--Kutta stability, and nonlinear stability
\cite{AbdiHojjati2011,AkbariHojjatiAbdi2023}, while recent implicit
multiderivative Runge--Kutta work highlights the implementation and
conditioning issues associated with higher derivative information
\cite{ChouchoulisSchuetz2024}.  Related second-derivative peer constructions
provide another multivalue route to high order and stiff stability
\cite{SharifiEtAl2026}.  General background on one-step, multistep, general
linear, and stiff integration methods can be found in
\cite{HairerNorsettWanner1993,HairerWanner1996,Butcher2016,Jackiewicz2009},
while early second-derivative formulas for stiff systems were developed by
Enright and Cash \cite{Enright1974,Cash1981}.

An important motivation comes from two-stage fourth-order Lax--Wendroff-type
time discretizations \cite{LiDu2016,YuanTang2023}.  In that setting, the known
state at the beginning of the step is input data rather than an unknown
stage.  This observation motivates the active-stage convention used here:
$Y_0=u^n$ is a known Hermite anchor and only $Y_1,\ldots,Y_s$ are counted as
active stages.  The method therefore solves for exactly $s$ unknown stage
states, although it uses information at $s+1$ nodes.  Each node contributes a
value--derivative pair, so the information available to the order conditions
is enriched without doubling the nonlinear stage vector.  The resulting
compactness is understood in this information-per-unknown sense: with $s$
unknown states the construction realizes order $2s$ together with
$L$-stability.  Among classical $s$-stage collocation methods based only on
$\Lop$, Gauss--Legendre has order $2s$ but is not $L$-stable, whereas Radau IIA
is $L$-stable but has order $2s-1$.

For stiff problems, the central objective is to combine high order with a
one-step stability function that is available throughout the entire family.
The present paper addresses the constructive problem
\[
  \boxed{\begin{gathered}
  s\text{ active unknown stages}\quad\Longrightarrow\quad
  \text{global order }2s,\\
  \text{with }L\text{-stability for every }s.
  \end{gathered}}
\]
Throughout this paper, $L$-stability is used in its standard one-step sense:
its stability function is $A$-stable and tends to zero for infinitely stiff
modes.  The stability function refers to the accepted map
$u^n\mapsto u^{n+1}$; this terminology does not assert contractivity or
uniform boundedness of the internal stages.

The initial anchor and the $s$ active nodes provide $s+1$ value--derivative
pairs.  Consequently, each stage row contains $2s+2$ coefficients.  The first
$2s$ Hermite moments impose global order $2s$ and leave two coefficients per
row.  The central constructive problem is to use these remaining degrees of
freedom to realize a prescribed stable rational function while preserving the
full set of moment equations and avoiding hidden poles in the coupled stage
system.

The target is the second-subdiagonal Pad\'{e} approximant
\[
  [s-1/s+1]_{e^z}.
\]
The first and second subdiagonals of the exponential Pad\'{e} table are
analytic and bounded by one in the left half-plane \cite{Ehle1973}; related
order-star and generalized Pad\'{e} stability results appear in
\cite{HairerNorsettWanner1978,Butcher2002}.  The nontrivial issue is not the
stability of the target rational function, but its exact realization within
the anchored moment family for arbitrary nodes.  This is achieved by proving
that the truncated products $Q_s(z)e^{c_jz}$ generate a polynomial basis.

The main contributions are as follows.
\begin{enumerate}[label=(\roman*)]
  \item The anchored formulation contains exactly $s$ unknown active stages,
  while the known value $Y_0=u^n$ acts only as a Hermite anchor.
  \item The $2s$ Hermite moment equations imply local defect
  $\mathcal O(\Delta t^{2s+1})$ and global order $2s$ for sufficiently small
  step sizes under standard smoothness assumptions.
  \item For every ordered real node set
  $0=c_0<c_1<\cdots<c_s=1$, a Pad\'{e}--Hermite basis theorem gives a unique
  constructive closure of the two free coefficients in every stage row.
  \item The coupled stage determinant is exactly $Q_s$, so the realization
  introduces no hidden stage poles.
  \item The endpoint stability function is exactly
  $[s-1/s+1]_{e^z}$; hence the family is $L$-stable (and therefore
  $A$-stable) for every positive integer $s$.
  \item Exact symbolic verification, high-precision order tests, and
  fixed-stage comparisons with Gauss--Legendre and Radau IIA methods confirm
  the algebraic construction and illustrate the combined order $2s$ and
  $L$-stability at the same active-stage count.
\end{enumerate}

The numerical comparisons use the two closest classical collocation
benchmarks.  An $s$-stage Gauss--Legendre method has order $2s$ and is
$A$-stable but not $L$-stable, whereas an $s$-stage Radau IIA method is
$L$-stable (and hence $A$-stable) but has order $2s-1$
\cite{HairerWanner1996}.  The present
family is designed to occupy the missing corner of this comparison: order
$2s$ together with $L$-stability, using $s$ active states and first
total-derivative information.

The construction separates accuracy, stability, and node placement.  Once the
Pad\'{e} closure is imposed, scalar endpoint stability no longer depends on
the active-node locations; the nodes may therefore be selected according to
conditioning, stage placement, or nonlinear-solver considerations.

The remainder of the paper is organized as follows.  \Cref{sec:family}
introduces the method and its Hermite moment conditions, and
\Cref{sec:convergence,sec:stability-background} establish order $2s$ and the
Pad\'{e} stability target.  \Cref{sec:pade-closure} proves the unique closure,
the exact stage-determinant identity, and $L$-stability for every $s$.
\Cref{sec:computation,sec:sequential,sec:first-four} address coefficient
construction, the relation to sequential schemes, and the first four members.
\Cref{sec:numerics} presents symbolic, accuracy, and stiff-problem tests.
Finally, \Cref{sec:discussion,sec:conclusions} discuss practical considerations
and conclude the paper.

\section{The anchored two-derivative family}
\label{sec:family}

Consider the autonomous initial-value problem
\begin{equation}
  u_t=\Lop(u),
  \qquad u(t_0)=u_0,
  \qquad u(t)\in\RR^d,
  \label{eq:ivp}
\end{equation}
where $\Lop$ is sufficiently smooth.  Its first total time derivative along a
solution is
\begin{equation}
  \Lt(u)=\frac{d}{dt}\Lop(u)=\Lop_u(u)\Lop(u).
  \label{eq:Ldot}
\end{equation}
For a nonautonomous problem $u_t=\Lop(t,u)$, one replaces
\eqref{eq:Ldot} by
\[
  \Lt(t,u)=\Lop_t(t,u)+\Lop_u(t,u)\Lop(t,u).
\]

Let $t_{n+1}=t_n+\dt$.  The known initial state is
\begin{equation}
  Y_0=u^n,
  \qquad c_0=0,
  \label{eq:anchor}
\end{equation}
and is called the \emph{initial Hermite anchor}.  Choose $s$ active nodes
\begin{equation}
  0<c_1<c_2<\cdots<c_s=1,
  \label{eq:nodes}
\end{equation}
with
\[
  Y_i\approx u(t_n+c_i\dt),
  \qquad i=1,\ldots,s.
\]
The general anchored method is
\begin{equation}
  Y_i
  =Y_0
  +\dt\sum_{j=0}^s a_{ij}\Lop(Y_j)
  +\dt^2\sum_{j=0}^s \widehat a_{ij}\Lt(Y_j),
  \qquad i=1,\ldots,s,
  \label{eq:anchored-method}
\end{equation}
with endpoint update
\begin{equation}
  u^{n+1}=Y_s.
  \label{eq:update}
\end{equation}
All $Y_1,\ldots,Y_s$ are active stages; $Y_0$ is known input data.

\subsection{Derivation of the Hermite moment conditions}

The moment conditions are obtained by matching the Taylor expansion of the
exact solution at every active stage.  Let
\[
  U_j=u(t_n+c_j\dt),
  \qquad j=0,\ldots,s,
\]
denote the exact solution values at the information nodes.  Along the exact
solution,
\[
  \Lop(U_j)=u'(t_n+c_j\dt),
  \qquad
  \Lt(U_j)=u''(t_n+c_j\dt).
\]
After inserting the exact stage values into the $i$th row of
\eqref{eq:anchored-method}, define the corresponding stage defect by
\begin{equation}
\begin{aligned}
  d_i={}&U_i-U_0
  -\dt\sum_{j=0}^s a_{ij}u'(t_n+c_j\dt)
  -\dt^2\sum_{j=0}^s\widehat a_{ij}u''(t_n+c_j\dt).
\end{aligned}
\label{eq:defect}
\end{equation}
Taylor expansion about $t_n$ gives
\begin{equation}
  U_i-U_0
  =\sum_{m=1}^{2s}
    \frac{c_i^m}{m!}\dt^m u^{(m)}(t_n)
  +\OO(\dt^{2s+1}),
  \label{eq:taylor-state}
\end{equation}
\begin{equation}
  \dt\,u'(t_n+c_j\dt)
  =\sum_{m=1}^{2s}
    \frac{c_j^{m-1}}{(m-1)!}\dt^m u^{(m)}(t_n)
  +\OO(\dt^{2s+1}),
  \label{eq:taylor-first-derivative}
\end{equation}
and
\begin{equation}
  \dt^2u''(t_n+c_j\dt)
  =\sum_{m=2}^{2s}
    \frac{c_j^{m-2}}{(m-2)!}\dt^m u^{(m)}(t_n)
  +\OO(\dt^{2s+1}).
  \label{eq:taylor-second-derivative}
\end{equation}
Substitution of \eqref{eq:taylor-state}--\eqref{eq:taylor-second-derivative}
into \eqref{eq:defect} yields
\begin{equation}
\begin{aligned}
  d_i
  ={}&\sum_{m=1}^{2s}
  \Bigg[
    \frac{c_i^m}{m!}
    -\frac{1}{(m-1)!}\sum_{j=0}^s a_{ij}c_j^{m-1} 
    -\frac{1}{(m-2)!}\sum_{j=0}^s
      \widehat a_{ij}c_j^{m-2}
  \Bigg]\dt^m u^{(m)}(t_n)
  +\OO(\dt^{2s+1}),
\end{aligned}
\label{eq:defect-expansion}
\end{equation}
where the term containing $(m-2)!$ is absent for $m=1$.  To obtain
$d_i=\OO(\dt^{2s+1})$, the coefficient of every
$\dt^m u^{(m)}(t_n)$ must vanish for $m=1,\ldots,2s$.  Multiplying the
resulting identity by $(m-1)!$ gives the Hermite moment conditions
\begin{equation}
  \boxed{
  \sum_{j=0}^s a_{ij}c_j^{m-1}
  +(m-1)\sum_{j=0}^s\widehat a_{ij}c_j^{m-2}
  =\frac{c_i^m}{m},
  \qquad m=1,\ldots,2s.
  }
  \label{eq:moments}
\end{equation}
For $m=1$, the second sum is understood to be zero.  Thus,
\eqref{eq:moments} is precisely the condition that the stage formula reproduce
the Taylor expansion of the exact solution through degree $2s$.  Every row
contains the $2s+2$ coefficients
\[
  a_{i0},\ldots,a_{is},
  \qquad
  \widehat a_{i0},\ldots,\widehat a_{is}.
\]
Consequently, the $2s$ moment equations leave two degrees of freedom per row
whenever the moment matrix has full row rank.  The Pad\'e closure developed
below proves solvability and uniquely fixes these two remaining degrees of
freedom.

\section{Consistency, local solvability, and convergence}
\label{sec:convergence}

The preceding Taylor calculation immediately gives the following stage
consistency result.

\begin{theorem}[Stage consistency]
\label{thm:stage-consistency}
If the coefficients satisfy \eqref{eq:moments} and the exact solution has
$2s+1$ bounded time derivatives on the step, then
\[
  d_i=\OO(\dt^{2s+1}),
  \qquad i=1,\ldots,s.
\]
\end{theorem}

\begin{proof}
By \eqref{eq:defect-expansion}, the moment conditions cancel all terms of
orders $\dt,\ldots,\dt^{2s}$.  The first uncancelled term is therefore of
order $\dt^{2s+1}$.
\end{proof}

\begin{theorem}[Local solvability and global order]
\label{thm:convergence}
Assume that $\Lop$ and $\Lt$ are sufficiently smooth in a neighborhood of the
exact solution and that the exact solution has $2s+1$ bounded time
derivatives on $[0,T]$.  For all sufficiently small $\dt$, the stage system
\eqref{eq:anchored-method} has a locally unique solution branch satisfying
$Y_i\to u^n$ as $\dt\to0$.  A step started from exact data satisfies
\[
  u^{n+1}-u(t_{n+1})=\OO(\dt^{2s+1}),
\]
and, for exact initial data,
\[
  \max_{0\le n\le T/\dt}\|u^n-u(t_n)\|=\OO(\dt^{2s}).
\]
\end{theorem}

\begin{proof}
Set $p=2s$ and collect the active stages in
$Y=(Y_1,\ldots,Y_s)^T\in(\RR^d)^s$.  For an anchor value $v$ and step size
$h\ge0$, define the residual map componentwise by
\begin{equation}
\begin{aligned}
  F_i(Z;v,h)
  ={}&Z_i-v
  -h\left(a_{i0}\Lop(v)+\sum_{j=1}^s a_{ij}\Lop(Z_j)\right)
-h^2\left(\widehat a_{i0}\Lt(v)
  +\sum_{j=1}^s\widehat a_{ij}\Lt(Z_j)\right),
  \qquad i=1,\ldots,s.
\end{aligned}
\label{eq:stage-residual-map}
\end{equation}
At $h=0$, $F(\one v;v,0)=0$ and
$D_ZF(\one v;v,0)=I_{sd}$.  The implicit-function theorem therefore gives a
locally unique smooth stage map $\Psi_h(v)$ with
$F(\Psi_h(v);v,h)=0$ and $\Psi_0(v)=\one v$.  On a compact neighborhood of
the exact solution,
\[
  D_ZF(\Psi_h(v);v,h)=I_{sd}+\OO(h),
\]
uniformly in $v$.  Hence this Jacobian is invertible for sufficiently small
$h$ and its inverse is $I_{sd}+\OO(h)$.

For a step starting from $v=u(t_n)$, let
\[
  U^n=\bigl(u(t_n+c_1h),\ldots,u(t_n+c_sh)\bigr)^T.
\]
The stage defects satisfy
$F(U^n;u(t_n),h)=d^n$ with $\|d^n\|\le K_dh^{p+1}$ by
\cref{thm:stage-consistency}.  Applying the local inverse map to the residual
perturbation gives
\begin{equation}
  \|U^n-\Psi_h(u(t_n))\|\le K_J\|d^n\|
  \le K_JK_dh^{p+1}.
\label{eq:stage-error-bound}
\end{equation}
Because $c_s=1$, the endpoint component $\Phi_h(v):=\Psi_{h,s}(v)$ therefore
satisfies
\begin{equation}
  \tau_{n+1}:=\Phi_h(u(t_n))-u(t_{n+1}),
  \qquad \|\tau_{n+1}\|\le K_\tau h^{p+1}.
\label{eq:local-truncation-error}
\end{equation}

It remains to control perturbations of the starting value.  Differentiating
$F(\Psi_h(v);v,h)=0$ gives
\[
  D_v\Psi_h(v)
  =-\left[D_ZF(\Psi_h(v);v,h)\right]^{-1}
    D_vF(\Psi_h(v);v,h).
\]
Uniform smoothness yields
$D_vF=-\one\otimes I_d+\OO(h)$, and consequently
$D_v\Psi_h=\one\otimes I_d+\OO(h)$.  Selecting the endpoint component gives
$D\Phi_h(v)=I_d+\OO(h)$, so the mean-value theorem implies
\begin{equation}
  \|\Phi_h(v)-\Phi_h(w)\|\le(1+Ch)\|v-w\|
\label{eq:one-step-lipschitz}
\end{equation}
for $v,w$ in the relevant neighborhood and sufficiently small $h$.

With $e_n=u^n-u(t_n)$, equations
\eqref{eq:local-truncation-error}--\eqref{eq:one-step-lipschitz} give
\[
  \|e_{n+1}\|\le(1+Ch)\|e_n\|+K_\tau h^{p+1}.
\]
The discrete Gronwall estimate on $nh\le T$, together with $e_0=0$, yields
$\max_n\|e_n\|\le C_T h^p$.  The same bound keeps the numerical values inside
the neighborhood on which the local branch and uniform estimates are valid,
closing the continuation argument.
\end{proof}

Thus the moment family achieves the desired accuracy statement:
\[
  \boxed{\text{$s$ active stages and global order $2s$.}}
\]
The remaining issue is to close the two free coefficients per row so that the
whole family has an all-$s$ stability theorem.

\section{Linear stability and the Pad\'e target}
\label{sec:stability-background}

Partition the coefficient arrays into the anchor columns and the active-stage
blocks:
\[
  a_0=(a_{10},\ldots,a_{s0})^T,
  \qquad A=(a_{ij})_{i,j=1}^s,
\]
\[
  \widehat a_0=(\widehat a_{10},\ldots,\widehat a_{s0})^T,
  \qquad \widehat A=(\widehat a_{ij})_{i,j=1}^s.
\]
For the Dahlquist test equation
\[
  u_t=\lambda u,
  \qquad z=\lambda\dt,
\]
one has
\[
  \Lop(u)=\lambda u,
  \qquad
  \Lt(u)=\lambda^2u.
\]
Consequently, after writing
\[
  Y=(Y_1,\ldots,Y_s)^T,
\]
the active-stage equations become
\begin{equation}
  M_s(z)Y=b_s(z)u^n,
  \qquad
  M_s(z):=I-zA-z^2\widehat A,
  \qquad
  b_s(z):=\one+za_0+z^2\widehat a_0.
  \label{eq:dahlquist-stage}
\end{equation}
Since $M_s(0)=I$, the matrix $M_s(z)$ is invertible for all sufficiently small
$z$.  The endpoint stage $Y_s$ is the numerical solution at $t_{n+1}$, so the
corresponding stability function is
\begin{equation}
  R_s(z)=\es^TM_s(z)^{-1}b_s(z).
  \label{eq:general-stability}
\end{equation}

\subsection{Accuracy inherited from the moment conditions}

The relation between the moment equations and the Taylor expansion of
$R_s$ can be seen directly at the stage level.  Define the vector of exact
stage ratios for the Dahlquist equation by
\[
  E(z)=
  \begin{pmatrix}
    e^{c_1z}\\ \vdots\\ e^{c_sz}
  \end{pmatrix},
  \qquad e^{c_0z}=1.
\]
For the $i$th row, substitution of the exact ratios into the numerical stage
formula produces the residual
\begin{equation}
\begin{aligned}
  r_i(z)
  :={}&e^{c_i z}-1
  -z\sum_{j=0}^s a_{ij}e^{c_jz}
  -z^2\sum_{j=0}^s\widehat a_{ij}e^{c_jz}.
\end{aligned}
\label{eq:linear-stage-residual}
\end{equation}
Expanding the exponentials at $z=0$ gives
\begin{equation}
\begin{aligned}
  r_i(z)
  ={}&\sum_{m=1}^{\infty}
  \Bigg[
    \frac{c_i^m}{m!}
    -\frac{1}{(m-1)!}\sum_{j=0}^s a_{ij}c_j^{m-1}
    -\frac{1}{(m-2)!}\sum_{j=0}^s\widehat a_{ij}c_j^{m-2}
  \Bigg]z^m,
\end{aligned}
\label{eq:linear-residual-expansion}
\end{equation}
where the last term is omitted for $m=1$.  For $m\ge2$, the coefficient in
brackets can be written as
\[
  \frac{1}{(m-1)!}
  \left[
    \frac{c_i^m}{m}
    -\sum_{j=0}^s a_{ij}c_j^{m-1}
    -(m-1)\sum_{j=0}^s\widehat a_{ij}c_j^{m-2}
  \right].
\]
The same expression, with the second-derivative sum absent, applies when
$m=1$.  Hence the moment conditions \eqref{eq:moments} make every coefficient
of $z^m$ vanish for $m=1,\ldots,2s$, and therefore
\begin{equation}
  r_i(z)=\OO(z^{2s+1}),
  \qquad i=1,\ldots,s.
  \label{eq:linear-stage-residual-order}
\end{equation}
In vector form, \eqref{eq:linear-stage-residual} is precisely
\begin{equation}
  M_s(z)E(z)-b_s(z)=r(z),
  \qquad r(z)=\OO(z^{2s+1}).
  \label{eq:ME-residual}
\end{equation}
Let
\[
  y(z):=\frac{Y}{u^n}=M_s(z)^{-1}b_s(z)
\]
denote the vector of numerical stage ratios.  Since $M_s(0)=I$, its inverse is
analytic and bounded in a neighborhood of the origin.  Subtracting
\eqref{eq:ME-residual} from the numerical stage equation gives
\[
  y(z)-E(z)=-M_s(z)^{-1}r(z)=\OO(z^{2s+1}).
\]
Because the endpoint node is $c_s=1$, the last component yields
\begin{equation}
  R_s(z)=e^z+\OO(z^{2s+1}),
  \qquad z\to0.
  \label{eq:stability-order-from-moments}
\end{equation}
Thus, the moment conditions guarantee that the stability function matches the
exponential through degree $2s$.  This is a local condition at $z=0$ only.
It does not control poles or the magnitude of $R_s(z)$ throughout the left
half-plane and therefore does not, by itself, imply $A$- or $L$-stability.

\subsection{The second-subdiagonal Pad\'e target}

For nonnegative integers $p$ and $q$, the notation
\[
  [p/q]_{e^z}
\]
denotes the Pad\'e approximant of type $(p,q)$ to the exponential function.
It is the rational function $P_{p,q}(z)/Q_{p,q}(z)$ satisfying
\[
  \deg P_{p,q}\le p,
  \qquad
  \deg Q_{p,q}\le q,
  \qquad
  Q_{p,q}(0)=1,
\]
and whose Maclaurin expansion agrees with that of $e^z$ through degree
$p+q$.  Equivalently, its defining approximation condition is
\begin{equation}
  Q_{p,q}(z)e^z-P_{p,q}(z)
  =\OO\!\left(z^{p+q+1}\right),
  \qquad z\to0.
  \label{eq:general-pade-definition}
\end{equation}

The construction below uses the second-subdiagonal member obtained by taking
\[
  p=s-1,
  \qquad q=s+1.
\]
Thus
\begin{equation}
  R_s^\star(z)
  =\frac{P_s(z)}{Q_s(z)}
  =[\,s-1\,/\,s+1\,]_{e^z},
  \qquad Q_s(0)=1.
  \label{eq:pade-target}
\end{equation}
The term ``second subdiagonal'' refers to the degree difference
$q-p=2$.  Since
\[
  p+q=(s-1)+(s+1)=2s,
\]
the defining Pad\'e condition \eqref{eq:general-pade-definition} becomes
\begin{equation}
  Q_s(z)e^z-P_s(z)=\OO(z^{2s+1}).
  \label{eq:pade-order}
\end{equation}
In particular,
\[
  \frac{P_s(z)}{Q_s(z)}
  =e^z+\OO(z^{2s+1})
\]
near the origin, because $Q_s(0)=1$.  The Pad\'e target therefore has exactly
the local approximation order required by \eqref{eq:stability-order-from-moments}.

With the normalization $Q_s(0)=1$, the numerator and denominator can be
written explicitly as
\begin{equation}
  P_s(z)=\sum_{k=0}^{s-1}
  \frac{(2s-k)!(s-1)!}{(2s)!k!(s-1-k)!}\,z^k,
  \label{eq:P-explicit}
\end{equation}
\begin{equation}
  Q_s(z)=\sum_{k=0}^{s+1}
  \frac{(2s-k)!(s+1)!}{(2s)!k!(s+1-k)!}\,(-z)^k.
  \label{eq:Q-explicit}
\end{equation}
For completeness, the coefficient of $z^m$ in $Q_s(z)e^z$ is
\begin{equation}
  [z^m]\bigl(Q_s(z)e^z\bigr)
  =\sum_{k=0}^{\min\{m,s+1\}}
  \frac{(-1)^k(2s-k)!(s+1)!}
  {(2s)!k!(s+1-k)!(m-k)!}.
  \label{eq:Qexp-coefficient}
\end{equation}
The standard alternating binomial identity for this sum gives
\begin{equation}
  [z^m]\bigl(Q_s(z)e^z\bigr)
  =
  \begin{cases}
  \displaystyle
  \frac{(2s-m)!(s-1)!}
  {(2s)!m!(s-1-m)!},
  &0\le m\le s-1,\\[3mm]
  0,&s\le m\le2s.
  \end{cases}
  \label{eq:Qexp-coefficient-cancellation}
\end{equation}
The first line is exactly the coefficient of $z^m$ in $P_s$, while $P_s$
has no terms of degree $m\ge s$.  Hence all coefficients through degree
$2s$ in $Q_s e^z-P_s$ vanish, which verifies \eqref{eq:pade-order} directly.

The exponential Pad\'e table is normal, so $P_s$ and $Q_s$ have their stated
degrees and are coprime \cite{BakerGravesMorris1996}.  Moreover, the first and
second subdiagonals of the exponential Pad\'e table are analytic and bounded
by one in the closed left half-plane
\cite{Ehle1973,HairerNorsettWanner1978}.  Therefore the target
$P_s/Q_s$ is $A$-stable.  Since
\[
  \deg P_s=s-1,
  \qquad
  \deg Q_s=s+1,
\]
one also has
\[
  \frac{P_s(z)}{Q_s(z)}=\OO(z^{-2}),
  \qquad |z|\to\infty,
\]
and hence $P_s(z)/Q_s(z)\to0$ in the left half-plane.  The target is therefore
$L$-stable.  The remaining task is to prove that this rational function can be
realized by the anchored method while preserving all $2s$ moment conditions
and without introducing hidden stage poles.

\section{Pad\'e closure for arbitrary nodes}
\label{sec:pade-closure}

For a formal power series $f(z)=\sum_{m=0}^{\infty}f_mz^m$, let
\[
  \Ttr_N(f):=\sum_{m=0}^{N}f_mz^m
\]
denote Taylor truncation at degree $N$.  Define the stage polynomials
\begin{equation}
  \Pi_j(z)=\Ttr_{2s}\!\left(Q_s(z)e^{c_jz}\right),
  \qquad j=0,\ldots,s.
  \label{eq:Pi-def}
\end{equation}
Write
\[
  \Pi_j(z)=\sum_{r=0}^{2s}\pi_{j,r}z^r.
\]
At the anchor node $c_0=0$,
\[
  Q_s(z)e^{c_0z}=Q_s(z).
\]
Since $\deg Q_s=s+1\le2s$ for every $s\ge1$, Taylor truncation does not alter
this polynomial, and therefore
\begin{equation}
  \Pi_0(z)=Q_s(z).
  \label{eq:Pi0}
\end{equation}
At the endpoint node $c_s=1$, the Pad\'e relation \eqref{eq:pade-order} gives
\[
  Q_s(z)e^z=P_s(z)+\OO(z^{2s+1}).
\]
Thus the Taylor coefficients of $Q_s e^z$ through degree $2s$ coincide with
those of $P_s$.  Since $\deg P_s=s-1\le2s$, it follows that
\begin{equation}
  \Pi_s(z)
  =\Ttr_{2s}\!\left(Q_s(z)e^z\right)
  =P_s(z).
  \label{eq:Pis}
\end{equation}

\subsection{The highest retained coefficient}

Define
\begin{equation}
  \rho_s(c)=[z^{2s}]\left(Q_s(z)e^{cz}\right).
  \label{eq:rho-def}
\end{equation}

\begin{lemma}
\label{lem:rho}
For every positive integer $s$,
\begin{equation}
  \boxed{
  \rho_s(c)=\frac{c^{s-1}(c-1)^{s+1}}{(2s)!}.
  }
  \label{eq:rho-formula}
\end{equation}
Consequently,
\[
  \pi_{j,2s}=\rho_s(c_j).
\]
\end{lemma}

\begin{proof}
The expression $\rho_s(c)$ is a polynomial of degree $2s$ with leading
coefficient $1/(2s)!$.  Since $\deg Q_s=s+1$, the smallest possible power of
$c$ in \eqref{eq:rho-def} is $c^{2s-(s+1)}=c^{s-1}$, so $c=0$ is a zero of
multiplicity at least $s-1$.  Moreover,
\[
  \rho_s^{(r)}(1)
  =[z^{2s-r}]\left(Q_s(z)e^z\right),
  \qquad r=0,\ldots,s.
\]
The indices $2s-r$ range from $s$ to $2s$.  By \eqref{eq:pade-order} and
$\deg P_s=s-1$, all these coefficients vanish.  Thus $c=1$ is a zero of
multiplicity at least $s+1$.  The two multiplicities sum to $2s$, and the
leading coefficient determines \eqref{eq:rho-formula}.
\end{proof}

\subsection{A polynomial basis theorem}

Let
\[
  \omega(c)=\prod_{j=0}^s(c-c_j),
  \qquad
  \ell_j(c)=\frac{\omega(c)}{(c-c_j)\omega'(c_j)}.
\]
The cardinal Hermite derivative polynomial
\begin{equation}
  K_j(c)=(c-c_j)\ell_j(c)^2
  \label{eq:Kj}
\end{equation}
satisfies $K_j(c_k)=0$ and $K_j'(c_k)=\delta_{jk}$, with
\begin{equation}
  \lc(K_j)=\frac{1}{\omega'(c_j)^2}>0.
  \label{eq:K-leading}
\end{equation}

\begin{theorem}[Pad\'e--Hermite basis]
\label{thm:basis}
For every ordered node set
$0=c_0<c_1<\cdots<c_s=1$, the $2s+2$ polynomials
\begin{equation}
  \Pi_0,\ldots,\Pi_s,
  \quad z\Pi_0,\ldots,z\Pi_s
  \label{eq:basis-set}
\end{equation}
form a basis of $\PP_{2s+1}$.
\end{theorem}

\begin{proof}
Assume
\begin{equation}
  \sum_{j=0}^s\alpha_j\Pi_j(z)
  +z\sum_{j=0}^s\beta_j\Pi_j(z)=0.
  \label{eq:basis-relation}
\end{equation}
Since
\[
  \Pi_j(z)=Q_s(z)e^{c_jz}+\OO(z^{2s+1}),
\]
comparison through degree $2s$ in \eqref{eq:basis-relation}, followed by
division by $Q_s(0)=1$, gives
\begin{equation}
  \sum_{j=0}^s(\alpha_j+z\beta_j)e^{c_jz}
  =\OO(z^{2s+1}).
  \label{eq:exp-relation}
\end{equation}
Define a linear functional on $\PP_{2s+1}$ by
\[
  \Lambda(p)=\sum_{j=0}^s\alpha_jp(c_j)
  +\sum_{j=0}^s\beta_jp'(c_j).
\]
The coefficients of \eqref{eq:exp-relation} show that
$\Lambda(c^r)=0$ for $r=0,\ldots,2s$.  Therefore
\begin{equation}
  \Lambda(p)=\kappa\,[c^{2s+1}]p(c)
  \label{eq:leading-functional}
\end{equation}
for some scalar $\kappa$.  Applying \eqref{eq:leading-functional} to
$K_j$ gives
\begin{equation}
  \beta_j=\frac{\kappa}{\omega'(c_j)^2}.
  \label{eq:beta-sign}
\end{equation}
Thus all nonzero $\beta_j$ have the same sign.

The coefficient of $z^{2s+1}$ in \eqref{eq:basis-relation} is
\begin{equation}
  \sum_{j=0}^s\beta_j\pi_{j,2s}=0.
  \label{eq:top-relation}
\end{equation}
For $s\ge2$, \cref{lem:rho} shows that the endpoint contributions vanish and
all internal values $\pi_{j,2s}$ have the same nonzero sign
$(-1)^{s+1}$.  In view of \eqref{eq:beta-sign},
\eqref{eq:top-relation} is impossible unless $\kappa=0$.  For $s=1$,
$\pi_{0,2}=1/2$ and $\pi_{1,2}=0$, so the same conclusion follows.  Hence
$\Lambda=0$, and the independence of the Hermite value and derivative
functionals at distinct nodes gives $\alpha_j=\beta_j=0$ for every $j$.
There are $2s+2$ independent polynomials in the $(2s+2)$-dimensional space
$\PP_{2s+1}$.
\end{proof}

\subsection{Unique coefficient construction}

Since $\Pi_i(0)=Q_s(0)=1$, the polynomial
$(\Pi_i-Q_s)/z$ belongs to $\PP_{2s-1}$.  By \cref{thm:basis}, for every
$i=1,\ldots,s$ there are unique coefficients $a_{ij}$ and
$\widehat a_{ij}$ such that
\begin{equation}
  \boxed{
  \frac{\Pi_i(z)-Q_s(z)}{z}
  =\sum_{j=0}^s a_{ij}\Pi_j(z)
  +z\sum_{j=0}^s\widehat a_{ij}\Pi_j(z).
  }
  \label{eq:representation}
\end{equation}
Equivalently,
\begin{equation}
  \Pi_i(z)=Q_s(z)
  +z\sum_{j=0}^s a_{ij}\Pi_j(z)
  +z^2\sum_{j=0}^s\widehat a_{ij}\Pi_j(z).
  \label{eq:Pi-identity}
\end{equation}

\begin{theorem}[Moment preservation]
\label{thm:moment-preservation}
The coefficients defined by \eqref{eq:representation} satisfy all moment
conditions \eqref{eq:moments}.
\end{theorem}

\begin{proof}
Use
\[
  \Pi_j(z)=Q_s(z)e^{c_jz}+\OO(z^{2s+1})
\]
in \eqref{eq:Pi-identity}.  The result is
\[
  Q_s(z)\left(
  e^{c_i z}-1
  -z\sum_{j=0}^sa_{ij}e^{c_jz}
  -z^2\sum_{j=0}^s\widehat a_{ij}e^{c_jz}
  \right)
  =\OO(z^{2s+1}).
\]
Since $Q_s(0)=1$, the expression in parentheses is
$\OO(z^{2s+1})$.  Equating the coefficients of $z^m$ for
$m=1,\ldots,2s$ gives \eqref{eq:moments}.
\end{proof}

The two closure conditions that supplement the moments can be read directly
from the coefficients of $z^{2s+1}$ and $z^{2s+2}$ in
\eqref{eq:Pi-identity}:
\begin{equation}
  \sum_{j=0}^s a_{ij}\pi_{j,2s}
  +\sum_{j=0}^s\widehat a_{ij}\pi_{j,2s-1}=0,
  \label{eq:closure1}
\end{equation}
\begin{equation}
  \sum_{j=0}^s\widehat a_{ij}\pi_{j,2s}=0.
  \label{eq:closure2}
\end{equation}
Thus \eqref{eq:moments}, \eqref{eq:closure1}, and \eqref{eq:closure2}
form a square nonsingular linear system of $2s+2$ equations for every row.

\subsection{The stage determinant and absence of hidden poles}

Define
\begin{equation}
  M_s(z)=I-zA-z^2\widehat A,
  \qquad
  b_s(z)=\one+za_0+z^2\widehat a_0.
  \label{eq:M-b}
\end{equation}
Separating the anchor terms in \eqref{eq:Pi-identity} gives
\begin{equation}
  M_s(z)
  \begin{pmatrix}
    \Pi_1(z)\\ \vdots\\ \Pi_s(z)
  \end{pmatrix}
  =Q_s(z)b_s(z).
  \label{eq:M-Pi}
\end{equation}
This already identifies a rational stage solution, but it does not yet rule
out extra poles in $M_s^{-1}$ that cancel in the endpoint ratio.  The next
theorem excludes such hidden poles.

\begin{theorem}[Exact stage determinant]
\label{thm:determinant}
For the coefficients defined by \eqref{eq:representation},
\begin{equation}
  \boxed{
  \det\left(I-zA-z^2\widehat A\right)=Q_s(z).
  }
  \label{eq:det-Q}
\end{equation}
\end{theorem}

\begin{proof}
Introduce the polynomial matrix
\begin{equation}
  \mathcal K_s(z)=\left[-b_s(z)\ \middle|\ M_s(z)\right]
  \in\RR[z]^{s\times(s+1)}.
  \label{eq:K-matrix}
\end{equation}
Equation \eqref{eq:M-Pi} is
\begin{equation}
  \mathcal K_s(z)
  \begin{pmatrix}
    \Pi_0(z)\\ \Pi_1(z)\\ \vdots\\ \Pi_s(z)
  \end{pmatrix}=0.
  \label{eq:K-kernel}
\end{equation}
At $z=0$, $M_s(0)=I$, so $\mathcal K_s$ has rank $s$ over the rational
function field $\RR(z)$.  Let
$\Delta=(\Delta_0,\ldots,\Delta_s)^T$ be the signed vector of maximal minors,
normalized so that $\Delta_0=\det M_s$.  The cofactor identity gives
$\mathcal K_s\Delta=0$.  The nullspace over $\RR(z)$ is one-dimensional, so
\[
  \Delta_j(z)=S(z)\Pi_j(z),
  \qquad j=0,\ldots,s,
\]
for a rational function $S$.  Since $\Pi_0=Q_s$, $\Pi_s=P_s$, and $P_s,Q_s$
are coprime, the polynomials $\Pi_0,\ldots,\Pi_s$ have greatest common divisor
one.  Therefore $S$ is a polynomial.

Every maximal minor has degree at most $2s$, because every entry of
$\mathcal K_s$ has degree at most two.  By \cref{thm:basis}, at least one
$\Pi_j$ has degree exactly $2s$; otherwise all polynomials in
\eqref{eq:basis-set} would lie in $\PP_{2s}$ and could not be independent.
Hence $S$ is constant.  At $z=0$,
\[
  \Delta_0(0)=\det M_s(0)=1,
  \qquad \Pi_0(0)=Q_s(0)=1,
\]
so $S=1$.  Thus $\Delta_0=\Pi_0=Q_s$, proving \eqref{eq:det-Q}.
\end{proof}

\subsection{Pad\'e realization and stability for every $s$}

\begin{theorem}[$L$-stable Pad\'{e} realization]
\label{thm:main-stability}
Let $s\ge1$ and let
\[
  0=c_0<c_1<\cdots<c_s=1.
\]
Construct the coefficients by \eqref{eq:representation}.  Then the anchored
method \eqref{eq:anchored-method} satisfies all $2s$ moment conditions.  For
the Dahlquist equation, every active stage has the exact transfer function
\begin{equation}
  G_j(z):=\frac{Y_j}{u^n}=\frac{\Pi_j(z)}{Q_s(z)},
  \qquad j=1,\ldots,s,
  \label{eq:stage-transfer}
\end{equation}
and the endpoint stability function is
\begin{equation}
  \boxed{
  R_s(z)=G_s(z)=\frac{P_s(z)}{Q_s(z)}=[s-1/s+1]_{e^z}.
  }
  \label{eq:exact-R}
\end{equation}
Moreover,
\begin{equation}
  \det(I-zA-z^2\widehat A)=Q_s(z),
\end{equation}
so the stage system has no poles other than the zeros of $Q_s$.  The accepted
one-step map is $L$-stable (and therefore $A$-stable) for every positive
integer $s$.
\end{theorem}

\begin{proof}
Moment preservation follows from \cref{thm:moment-preservation}.  For every
$z$ such that $Q_s(z)\ne0$, \cref{thm:determinant} implies that the coupled
stage system has a unique solution.  Equation \eqref{eq:M-Pi} therefore gives
\eqref{eq:stage-transfer}.  Since $c_s=1$ and $\Pi_s=P_s$, the endpoint ratio
is exactly \eqref{eq:exact-R}.

The second subdiagonal of the exponential Pad\'{e} table is analytic and
bounded by one in $\operatorname{Re}z\le0$ \cite{Ehle1973}.  Hence
\[
  |R_s(z)|\le1,
  \qquad \operatorname{Re}z\le0,
\]
which verifies the $A$-stability part of the standard definition.  In
addition,
\[
  \deg P_s=s-1,
  \qquad \deg Q_s=s+1,
\]
so
\[
  R_s(z)=\mathcal O(z^{-2})\to0
\]
as $|z|\to\infty$ in the left half-plane.  Together, these two properties
prove $L$-stability.  Finally, the pole statement follows directly from
\cref{thm:determinant}.
\end{proof}

Combining \cref{thm:convergence,thm:main-stability} gives the central result:
\[
  \boxed{
  \begin{aligned}
  &s\text{ active unknown stages},\\
  &\text{global order }2s,\\
  &R_s(z)=[s-1/s+1]_{e^z},\\
  &L\text{-stability for every }s\ge1,\\
  &\det(I-zA-z^2\widehat A)=Q_s(z).
  \end{aligned}}
\]

\section{Coefficient computation and the role of the nodes}
\label{sec:computation}

The proof provides two equivalent computational routes.

\subsection{Polynomial-basis construction}

Let $B_s$ be the square coefficient matrix whose columns are the monomial
coefficient vectors of
\[
  \Pi_0,\ldots,\Pi_s,z\Pi_0,\ldots,z\Pi_s
\]
in $\PP_{2s+1}$.  For row $i$, let $r_i$ be the coefficient vector of
$(\Pi_i-Q_s)/z$.  Then
\begin{equation}
  B_s
  \begin{pmatrix}
  a_{i0}\\ \vdots\\ a_{is}\\
  \widehat a_{i0}\\ \vdots\\ \widehat a_{is}
  \end{pmatrix}
  =r_i.
  \label{eq:basis-linear-system}
\end{equation}
Theorem \ref{thm:basis} proves that $B_s$ is nonsingular.

\begin{algorithm}[t]
\caption{Pad\'e-closed anchored method of order $2s$}
\label{alg:construction}
\begin{algorithmic}[1]
\Require $s\ge1$ and nodes $0=c_0<c_1<\cdots<c_s=1$
\State Form $P_s,Q_s$ from \eqref{eq:P-explicit}--\eqref{eq:Q-explicit}
\For{$j=0,\ldots,s$}
  \State $\Pi_j\gets\Ttr_{2s}(Q_se^{c_jz})$
\EndFor
\State Form the basis matrix $B_s$ from
$\{\Pi_j,z\Pi_j\}_{j=0}^s$
\For{$i=1,\ldots,s$}
  \State $r_i\gets\coef((\Pi_i-Q_s)/z)$
  \State Solve \eqref{eq:basis-linear-system} for row $i$
\EndFor
\State Return the full arrays $A_{\rm full}$ and $\widehat A_{\rm full}$
\end{algorithmic}
\end{algorithm}

\subsection{Moment-plus-closure construction}

Alternatively, solve the $2s$ moments \eqref{eq:moments} together with the two
closure equations \eqref{eq:closure1}--\eqref{eq:closure2}.  This form makes
clear that the Pad\'e condition uses exactly the two degrees of freedom left by
accuracy.

\subsection{What the nodes do and do not control}

Theorem \ref{thm:main-stability} changes the node-selection problem.  Scalar
linear stability no longer distinguishes admissible node sets: every ordered
set gives the same stability function.  The internal nodes still influence:
\begin{itemize}
  \item the magnitudes and signs of the coefficients;
  \item the conditioning of \eqref{eq:basis-linear-system};
  \item the spatial location of internal physical states;
  \item the conditioning and convergence of the coupled nonlinear stage
  solve;
  \item stage error constants for non-Dahlquist problems.
\end{itemize}
For low order, equispaced nodes provide simple exact rational coefficients.  At
larger $s$, clustered nodes such as shifted Chebyshev--Lobatto points are often
a better starting point for numerical conditioning, but no universal
optimality claim is made.  In high-order implementations, the coefficient
construction should use exact arithmetic, arbitrary precision, or a scaled
orthogonal polynomial basis rather than an unscaled monomial solve.

\subsection{A conditioning illustration}

The node independence of the endpoint stability function does not imply node
independence of the coefficient construction.  To illustrate this point,
\cref{tab:node-conditioning} reports the two-norm condition number of the
monomial basis matrix $B_s$ in \eqref{eq:basis-linear-system} and the largest
coefficient magnitude for equispaced and shifted Chebyshev--Lobatto nodes.
The quantities were evaluated in 100-digit arithmetic.  The two choices are
nearly indistinguishable at very low order; from moderate order onward, the
clustered nodes reduce both the condition number and coefficient growth.  The
condition number is basis- and scaling-dependent, so the table is a
diagnostic rather than an optimality statement.

\begin{table}[t]
\centering
\small
\setlength{\tabcolsep}{3.5pt}
\renewcommand{\arraystretch}{1.08}
\begin{tabular}{crrrr}
\toprule
$s$ & \shortstack{$\log_{10}\kappa_2(B_s)$\\eq.} & \shortstack{$\log_{10}\kappa_2(B_s)$\\C--L} & \shortstack{$\log_{10}\|C\|_{\max}$\\eq.} & \shortstack{$\log_{10}\|C\|_{\max}$\\C--L}\\
\midrule
2 & 3.282 & 3.282 & 0.125 & 0.125\\
3 & 5.582 & 5.763 & 0.084 & 0.372\\
4 & 8.685 & 8.696 & 0.927 & 0.778\\
5 & 11.934 & 11.880 & 1.387 & 1.343\\
6 & 15.435 & 15.227 & 2.213 & 1.882\\
7 & 19.070 & 18.714 & 2.805 & 2.443\\
8 & 22.846 & 22.319 & 3.640 & 3.006\\
9 & 26.734 & 26.028 & 4.293 & 3.576\\
10 & 30.725 & 29.830 & 5.136 & 4.149\\
\bottomrule
\end{tabular}
\caption{Conditioning of the monomial construction for equispaced (eq.) and shifted Chebyshev--Lobatto (C--L) nodes. Here $C$ denotes the collection of all $a_{ij}$ and $\widehat a_{ij}$ coefficients.}
\label{tab:node-conditioning}
\end{table}

\section{Relation to sequential two-stage fourth-order schemes}
\label{sec:sequential}

The common-denominator Pad\'e construction is generally fully coupled.  It should be
distinguished from a genuinely sequential two-stage fourth-order method.  A
representative sequential scheme first computes the midpoint state
$Y_1=u^{n+1/2}$ from
\begin{equation}
  Y_1=Y_0+\frac{\dt}{4}
  \left(\Lop(Y_0)+\Lop(Y_1)\right)
  +\frac{\dt^2}{48}
  \left(\Lt(Y_0)-\Lt(Y_1)\right),
  \label{eq:sequential-stage1}
\end{equation}
and then computes $Y_2=u^{n+1}$ from
\begin{equation}
\begin{aligned}
  Y_2={}&Y_0+\dt\Bigg[
  \left(\frac16+4C+\frac12D\right)\Lop(Y_0)
  +\left(\frac23-8C+2D\right)\Lop(Y_1)
  +\left(\frac16+4C-\frac52D\right)\Lop(Y_2)
  \Bigg]\\
  &+\dt^2\left[
  C\Lt(Y_0)+D\Lt(Y_1)
  +\left(\frac12D-C\right)\Lt(Y_2)
  \right].
\end{aligned}
\label{eq:sequential-stage2}
\end{equation}
For $D=-C$ and
\[
  \frac{25-\sqrt{105}}{780}
  \le C\le
  \frac{25+\sqrt{105}}{780},
\]
this family is $L$-stable.  It has one midpoint solve followed by one endpoint
solve and is therefore a genuine two-active-stage method.

The Pad\'e-closed $s=2$ member in the present paper is a different method: it
is fully coupled, but it extends to every $s$ through a single proof.  The
distinction exposes a structural tradeoff.  Sequential sparsity can be
optimized at a fixed low order; order $2s$ and endpoint $L$-stability for every $s$ are
obtained here by accepting a coupled implicit stage system.

If strict lower triangularity is imposed in \eqref{eq:anchored-method}, the
first active stage can use only the anchor and itself, giving four
coefficients.  Such a two-node Hermite formula cannot be exact beyond cubic
integrands, so it cannot satisfy all $2s$ row moments when $s\ge3$.  This does
not rule out high-order sequential multiderivative methods based on different
global order conditions or deferred correction, but it rules out the present
full-row moment construction in a single lower-triangular sweep.

\section{The first four members}
\label{sec:first-four}

For illustration, take equispaced nodes $c_j=j/s$.  The exact stability
functions are listed in \cref{tab:pade-first-four}.  All four endpoint maps are
$L$-stable and satisfy $R_s(z)=e^z+\OO(z^{2s+1})$.

\begin{table}[t]
\centering
\small
\begin{tabular}{cclc}
\toprule
$s$ & order & $R_s(z)=P_s(z)/Q_s(z)$ & stiff limit\\
\midrule
1 & 2 & $\displaystyle{1}/{\frac{z^{2} - 2 z + 2}{2}}$ & $R_s(z)=\mathcal O(z^{-2})$\\[1em]
2 & 4 & $\displaystyle{\frac{z + 4}{4}}/{- \frac{z^{3} - 6 z^{2} + 18 z - 24}{24}}$ & $R_s(z)=\mathcal O(z^{-2})$\\[1em]
3 & 6 & $\displaystyle{\frac{z^{2} + 10 z + 30}{30}}/{\frac{z^{4} - 12 z^{3} + 72 z^{2} - 240 z + 360}{360}}$ & $R_s(z)=\mathcal O(z^{-2})$\\[1em]
4 & 8 & $\displaystyle{\frac{z^{3} + 18 z^{2} + 126 z + 336}{336}}/{- \frac{z^{5} - 20 z^{4} + 200 z^{3} - 1200 z^{2} + 4200 z - 6720}{6720}}$ & $R_s(z)=\mathcal O(z^{-2})$\\
\bottomrule
\end{tabular}
\caption{First four endpoint stability functions.}
\label{tab:pade-first-four}
\end{table}

For $s=1$, the coefficient arrays including the anchor column are
\[
  A_{\rm full}=\begin{pmatrix}0&1\end{pmatrix},
  \qquad
  \widehat A_{\rm full}=\begin{pmatrix}0&-\frac12\end{pmatrix}.
\]
Thus
\begin{equation}
  Y_1=Y_0+\dt\Lop(Y_1)-\frac{\dt^2}{2}\Lt(Y_1),
  \qquad u^{n+1}=Y_1,
  \label{eq:s1-method}
\end{equation}
with
\[
  R_1(z)=\frac{1}{1-z+z^2/2}=[0/2]_{e^z}.
\]

For $s=2$, choose $(c_0,c_1,c_2)=(0,1/2,1)$.  Then
\begin{equation}
  A_{\rm full}=
  \begin{pmatrix}
  -\frac5{96}&\frac{11}{12}&-\frac{35}{96}\\
  -\frac16&\frac43&-\frac16
  \end{pmatrix},
  \qquad
  \widehat A_{\rm full}=
  \begin{pmatrix}
  -\frac{11}{192}&0&\frac{17}{192}\\
  -\frac1{12}&0&\frac1{12}
  \end{pmatrix},
  \label{eq:s2-coefficients}
\end{equation}
and
\[
  R_2(z)=
  \frac{1+z/4}{1-3z/4+z^2/4-z^3/24}
  =[1/3]_{e^z}.
\]

For $s=3$, choose
$(c_0,c_1,c_2,c_3)=(0,1/3,2/3,1)$.  The exact coefficient arrays are
\begin{equation}
\begingroup
\setlength{\arraycolsep}{4pt}
  A_{\rm full}=
  \begin{pmatrix}
  \frac{532}{1215}&-\frac{53}{144}&-\frac8{45}&\frac{8579}{19440}\\[0.5em]
  \frac{946}{1215}&-\frac{43}{45}&-\frac8{45}&\frac{1241}{1215}\\[0.5em]
  \frac45&-\frac{81}{80}&0&\frac{97}{80}
  \end{pmatrix},
\endgroup
  \label{eq:s3-A-coefficients}
\end{equation}
\begin{equation}
\begingroup
\setlength{\arraycolsep}{4pt}
  \widehat A_{\rm full}=
  \begin{pmatrix}
  \frac{197}{4860}&\frac{11}{240}&-\frac{11}{60}&-\frac{185}{3888}\\[0.5em]
  \frac{98}{1215}&\frac19&-\frac49&-\frac{133}{1215}\\[0.5em]
  \frac1{12}&\frac9{80}&-\frac9{20}&-\frac{29}{240}
  \end{pmatrix}.
\endgroup
  \label{eq:s3-Ahat-coefficients}
\end{equation}
Its sixth-order stability function is
\[
  R_3(z)=
  \frac{1+z/3+z^2/30}
       {1-2z/3+z^2/5-z^3/30+z^4/360}
  =[2/4]_{e^z}.
\]

For $s=4$, choose
$(c_0,c_1,c_2,c_3,c_4)=(0,1/4,1/2,3/4,1)$.  The exact coefficient arrays are
\begin{equation}
\begingroup
\setlength{\arraycolsep}{2.2pt}
\renewcommand{\arraystretch}{1.35}
  A_{\rm full}=
  \begin{pmatrix}
  -\frac{24488549}{181923840}&\frac{33389}{2842560}&\frac{6217}{4480}&-\frac{1968661}{2842560}&-\frac{58633019}{181923840}\\
  -\frac{1354501}{1421280}&-\frac{16508}{44415}&\frac{218}{35}&-\frac{132388}{44415}&-\frac{2022731}{1421280}\\
  -\frac{8762751}{6737920}&-\frac{62409}{105280}&\frac{37341}{4480}&-\frac{401199}{105280}&-\frac{12673761}{6737920}\\
  -\frac{117097}{88830}&-\frac{27136}{44415}&\frac{296}{35}&-\frac{165376}{44415}&-\frac{160297}{88830}
  \end{pmatrix}.
\endgroup
  \label{eq:s4-A-coefficients}
\end{equation}
\begin{equation}
\begingroup
\setlength{\arraycolsep}{2.2pt}
\renewcommand{\arraystretch}{1.35}
  \widehat A_{\rm full}=
  \begin{pmatrix}
  -\frac{1622893}{121282560}&-\frac{733}{6016}&\frac{320463}{3368960}&\frac{40991}{210560}&\frac{2619139}{121282560}\\
  -\frac{66667}{947520}&-\frac{214}{423}&\frac{642}{1645}&\frac{12626}{14805}&\frac{30167}{315840}\\
  -\frac{1273701}{13475840}&-\frac{4041}{6016}&\frac{1749519}{3368960}&\frac{236163}{210560}&\frac{1701963}{13475840}\\
  -\frac{5669}{59220}&-\frac{32}{47}&\frac{864}{1645}&\frac{1888}{1645}&\frac{7397}{59220}
  \end{pmatrix}.
\endgroup
  \label{eq:s4-Ahat-coefficients}
\end{equation}
The corresponding eighth-order stability function is
\[
  R_4(z)=
  \frac{1+3z/8+3z^2/56+z^3/336}
       {1-5z/8+5z^2/28-5z^3/168+z^4/336-z^5/6720}
  =[3/5]_{e^z}.
\]

These four members make the all-$s$ construction concrete: each method uses
exactly $s$ unknown active stages, attains order $2s$, and has the
second-subdiagonal Pad\'{e} stability function with quadratic stiff decay.

\FloatBarrier
\section{Numerical verification and fixed-stage comparisons}
\label{sec:numerics}

The computations address two questions.  First, do the constructed members
actually realize $s$ active stages, order $2s$, and the predicted endpoint
stability function?  Second, does the combination of order $2s$ and
$L$-stability provide a visible advantage over the closest classical
collocation alternatives?

The accompanying generator produces all tables, figures, and CSV files.
Symbolic identities are checked with exact SymPy arithmetic.  The smooth
convergence test uses 100-decimal-digit mpmath arithmetic.  The nonlinear
stiff comparisons use full Newton iteration with analytical Jacobians and
stopping tolerance $10^{-12}$.

\subsection{Comparison methods and protocol}

The comparison uses the $s$-stage Gauss--Legendre and Radau IIA collocation
methods.  Their relevant properties are summarized in
\cref{tab:comparison-properties}.

\begin{table}[tbp]
\centering
\small
\setlength{\tabcolsep}{5pt}
\begin{tabular}{@{}lccc@{}}
\toprule
method ($s$ active stages) & order & $A$-stable & $L$-stable\\
\midrule
Gauss--Legendre & $2s$ & yes & no\\
Radau IIA & $2s-1$ & yes & yes\\
proposed method & $2s$ & yes & yes\\
\bottomrule
\end{tabular}
\caption{Properties relevant to the fixed-stage comparisons. Stability refers to the accepted one-step map.}
\label{tab:comparison-properties}
\end{table}

For every comparison, all three methods use the same number $s$ of unknown
stage states and the same number $N$ of uniform time steps.  This protocol
isolates the accuracy and stiff-damping consequences of enriching each stage
with the total-derivative information
\[
  \Lt(u)=\Lop_u(u)\Lop(u).
\]
For the nonlinear Kaps problem, total Newton iterations are also reported to
indicate the nonlinear-solve effort under the common implementation used in
the experiments.

\subsection{Exact verification of the construction}

For a $d$-dimensional ODE, the nonlinear unknown in one step is
\[
  Y=(Y_1,\ldots,Y_s)\in\mathbb R^{sd},
\]
whereas $Y_0=u^n$ is prescribed data.  Thus the formulation contains exactly
$s$ active stages.  For equispaced nodes and $s=1,\ldots,6$, exact arithmetic
verifies the basis rank, all moment equations, the polynomial representation,
the determinant identity, and the endpoint stability function.

\begin{table}[tbp]
\centering
\small
\setlength{\tabcolsep}{4pt}
\begin{tabular}{cccccc}
\toprule
$s$ & rank$(B_s)$ & moments & row identities & $\det M_s=Q_s$ & $R_s=P_s/Q_s$\\
\midrule
1 & 4/4 & yes & yes & yes & yes\\
2 & 6/6 & yes & yes & yes & yes\\
3 & 8/8 & yes & yes & yes & yes\\
4 & 10/10 & yes & yes & yes & yes\\
5 & 12/12 & yes & yes & yes & yes\\
6 & 14/14 & yes & yes & yes & yes\\
\bottomrule
\end{tabular}
\caption{Exact symbolic verification for equispaced nodes.}
\label{tab:exact-verification}
\end{table}

\subsection{Order $2s$ on a smooth nonlinear problem}

The classical order is tested on
\begin{equation}
  u_t=-u^2,
  \qquad u(0)=1,
  \qquad u(t)=\frac{1}{1+t},
  \label{eq:nonlinear-test}
\end{equation}
for which $\Lop(u)=-u^2$ and $\Lt(u)=2u^3$.  The equispaced
$s=1,2,3,4$ members are integrated to $T=1$ with
$N=4,8,16,32,64$ uniform steps.  The observed rates approach
$2,4,6,$ and $8$, respectively.

\begin{table}[tbp]
\centering
\small
\begin{tabular}{ccccc}
\toprule
$s$ & $N$ & $h$ & error & observed order\\
\midrule
1 & 4 & $1/4$ & $5.5719\times10^{-3}$ & --\\
1 & 8 & $1/8$ & $1.6348\times10^{-3}$ & 1.76905\\
1 & 16 & $1/16$ & $4.4566\times10^{-4}$ & 1.87508\\
1 & 32 & $1/32$ & $1.1655\times10^{-4}$ & 1.93499\\
1 & 64 & $1/64$ & $2.9815\times10^{-5}$ & 1.96685\\
2 & 4 & $1/4$ & $6.0109\times10^{-5}$ & --\\
2 & 8 & $1/8$ & $4.0885\times10^{-6}$ & 3.87793\\
2 & 16 & $1/16$ & $2.6666\times10^{-7}$ & 3.93851\\
2 & 32 & $1/32$ & $1.7023\times10^{-8}$ & 3.96943\\
2 & 64 & $1/64$ & $1.0752\times10^{-9}$ & 3.98480\\
3 & 4 & $1/4$ & $6.6449\times10^{-7}$ & --\\
3 & 8 & $1/8$ & $1.1383\times10^{-8}$ & 5.86728\\
3 & 16 & $1/16$ & $1.8539\times10^{-10}$ & 5.94016\\
3 & 32 & $1/32$ & $2.9531\times10^{-12}$ & 5.97218\\
3 & 64 & $1/64$ & $4.6571\times10^{-14}$ & 5.98667\\
4 & 4 & $1/4$ & $7.6269\times10^{-9}$ & --\\
4 & 8 & $1/8$ & $3.4487\times10^{-11}$ & 7.78889\\
4 & 16 & $1/16$ & $1.4204\times10^{-13}$ & 7.92361\\
4 & 32 & $1/32$ & $5.6630\times10^{-16}$ & 7.97053\\
4 & 64 & $1/64$ & $2.2312\times10^{-18}$ & 7.98758\\
\bottomrule
\end{tabular}
\caption{High-precision convergence for $u_t=-u^2$ at $T=1$.}
\label{tab:nonlinear-convergence}
\end{table}

\begin{figure}[tbp]
  \centering
  \includegraphics[width=0.6\textwidth]{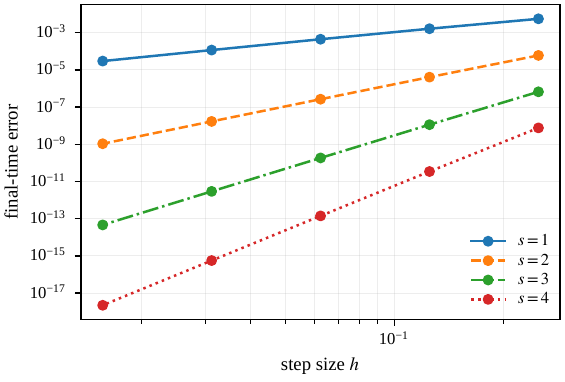}
  \caption{High-precision convergence for \eqref{eq:nonlinear-test}.  The
  asymptotic slopes approach $2s$ for $s=1,2,3,4$.}
  \label{fig:nonlinear-convergence}
\end{figure}

\subsection{$L$-stability diagnostics}

The stability theorem is analytical: the exact identity
\[
  R_s(z)=[s-1/s+1]_{e^z}
\]
and the classical Pad\'{e} theorem verify the $A$-stability part of the
standard definition, while $\deg Q_s-\deg P_s=2$ gives
$R_s(z)\to0$ as $|z|\to\infty$ in the left half-plane.  Together,
these properties establish $L$-stability.  The following computations diagnose
the implementation rather than replace the proof.  For $s=1,\ldots,6$, all
computed roots of $Q_s$ lie in the open right half-plane, a dense
imaginary-axis scan gives $\max|R_s(iy)|=1$ to roundoff, and the negative-axis
values exhibit the predicted quadratic decay.

\begin{table}[tbp]
\centering
\small
\begin{tabular}{ccccc}
\toprule
$s$ & $\min\Re\{Q_s=0\}$ & sampled $\max|R_s(iy)|$ & excess over one & $\lim x^2|R_s(-x)|$\\
\midrule
1 & 1.00000000 & 1.0000000000000000 & 0.00e+00 & 2\\
2 & 1.68709159 & 1.0000000000000000 & 0.00e+00 & 6\\
3 & 2.22098003 & 1.0000000000000002 & 2.22e-16 & 12\\
4 & 2.66473152 & 1.0000000000000002 & 2.22e-16 & 20\\
5 & 3.04820604 & 1.0000000000000002 & 2.22e-16 & 30\\
6 & 3.38811794 & 1.0000000000000002 & 2.22e-16 & 42\\
\bottomrule
\end{tabular}
\caption{Endpoint stability diagnostics.}
\label{tab:stability-diagnostics}
\end{table}

\begin{figure}[tbp]
  \centering
  \begin{subfigure}[t]{0.49\textwidth}
    \centering
    \includegraphics[width=\linewidth]{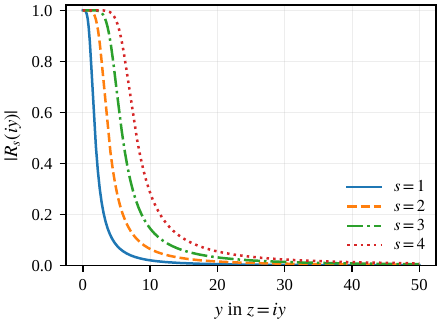}
    \caption{Endpoint modulus on the imaginary axis for $s=1,2,3,4$.}
    \label{fig:imag-axis}
  \end{subfigure}\hfill
  \begin{subfigure}[t]{0.49\textwidth}
    \centering
    \includegraphics[width=\linewidth]{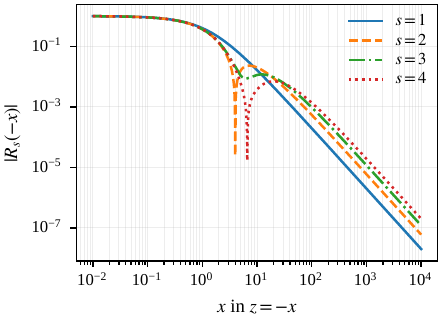}
    \caption{Endpoint decay on the negative real axis; each curve tends to
    zero as $\mathcal O(x^{-2})$.}
    \label{fig:negative-axis}
  \end{subfigure}
  \caption{Endpoint stability diagnostics for the first four members.}
  \label{fig:endpoint-stability-diagnostics}
\end{figure}

\subsection{A two-mode stiff decay problem}

Consider the diagonal system
\begin{equation}
  y'(t)=
  \begin{pmatrix}-1&0\\0&-1000\end{pmatrix}y(t),
  \qquad y(0)=\begin{pmatrix}1\\1\end{pmatrix},
  \qquad 0\le t\le1.
  \label{eq:two-mode}
\end{equation}
The slow component tests classical accuracy, while the fast component tests
stiff damping.  Gauss collocation has the same order $2s$ as the proposed
method, but its stability function tends to a nonzero value in magnitude as
$z\to-\infty$.  Radau IIA damps the fast component, but its classical order at
the same stage count is $2s-1$.

\begin{table}[tbp]
\centering
\small
\begin{tabular}{cclrr}
\toprule
$s$ & $N$ & method & total error & $|y_{\rm fast}(1)|$\\
\midrule
2 & 1 & Proposed & $5.325\times10^{-4}$ & $5.940\times10^{-6}$\\
2 & 1 & Gauss & $9.881\times10^{-1}$ & $9.881\times10^{-1}$\\
2 & 1 & Radau IIA & $4.243\times10^{-3}$ & $1.986\times10^{-3}$\\
\addlinespace
2 & 2 & Proposed & $3.938\times10^{-5}$ & $5.534\times10^{-10}$\\
2 & 2 & Gauss & $9.531\times10^{-1}$ & $9.531\times10^{-1}$\\
2 & 2 & Radau IIA & $5.700\times10^{-4}$ & $1.556\times10^{-5}$\\
\addlinespace
2 & 4 & Proposed & $2.706\times10^{-6}$ & $7.234\times10^{-17}$\\
2 & 4 & Gauss & $8.253\times10^{-1}$ & $8.253\times10^{-1}$\\
2 & 4 & Radau IIA & $7.505\times10^{-5}$ & $3.661\times10^{-9}$\\
\addlinespace
2 & 8 & Proposed & $1.777\times10^{-7}$ & $2.483\times10^{-28}$\\
2 & 8 & Gauss & $4.639\times10^{-1}$ & $4.639\times10^{-1}$\\
2 & 8 & Radau IIA & $9.664\times10^{-6}$ & $2.740\times10^{-15}$\\
\addlinespace
3 & 1 & Proposed & $1.174\times10^{-5}$ & $1.174\times10^{-5}$\\
3 & 1 & Gauss & $9.763\times10^{-1}$ & $9.763\times10^{-1}$\\
3 & 1 & Radau IIA & $2.949\times10^{-3}$ & $2.949\times10^{-3}$\\
\addlinespace
3 & 2 & Proposed & $6.631\times10^{-8}$ & $2.110\times10^{-9}$\\
3 & 2 & Gauss & $9.085\times10^{-1}$ & $9.085\times10^{-1}$\\
3 & 2 & Radau IIA & $3.363\times10^{-5}$ & $3.363\times10^{-5}$\\
\addlinespace
3 & 4 & Proposed & $1.107\times10^{-9}$ & $9.545\times10^{-16}$\\
3 & 4 & Gauss & $6.811\times10^{-1}$ & $6.811\times10^{-1}$\\
3 & 4 & Radau IIA & $4.794\times10^{-8}$ & $1.579\times10^{-8}$\\
\addlinespace
3 & 8 & Proposed & $1.791\times10^{-11}$ & $2.931\times10^{-26}$\\
3 & 8 & Gauss & $2.153\times10^{-1}$ & $2.153\times10^{-1}$\\
3 & 8 & Radau IIA & $1.527\times10^{-9}$ & $3.694\times10^{-14}$\\
\addlinespace
\bottomrule
\end{tabular}
\caption{Two-mode stiff decay with eigenvalues $-1$ and $-1000$. The exact fast component at $T=1$ is negligible.}
\label{tab:two-mode-comparison}
\end{table}

The proposed method suppresses the fast mode almost immediately while
retaining the higher slow-mode accuracy.  For example, with $s=3$ and only
two time steps, its endpoint error is below $7\times10^{-8}$, compared with
approximately $9\times10^{-1}$ for Gauss and $4\times10^{-5}$ for Radau IIA.

\begin{figure}[tbp]
  \centering
  \includegraphics[width=0.6\textwidth]{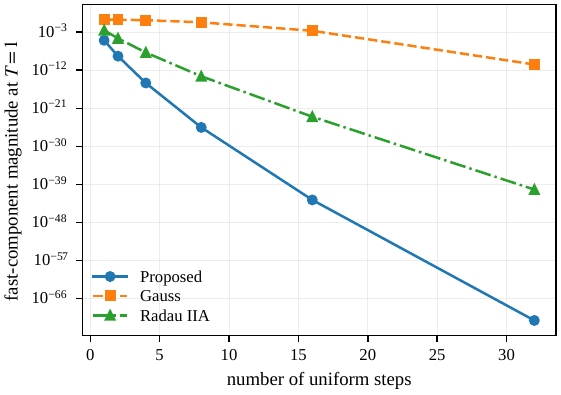}
  \caption{Fast-component damping for the $s=3$ methods applied to
  \eqref{eq:two-mode}.}
  \label{fig:two-mode-decay}
\end{figure}

\subsection{The nonlinear Kaps stiff problem}

A more demanding nonlinear comparison uses the classical Kaps problem
\cite{KapsRentrop1979}:
\begin{equation}
\begin{aligned}
  \varepsilon y_1'&=-(1+2\varepsilon)y_1+y_2^2,\\
  y_2'&=y_1-y_2-y_2^2,
\end{aligned}
\qquad
  y_1(0)=y_2(0)=1,
\label{eq:kaps}
\end{equation}
whose exact solution is
\[
  y_1(t)=e^{-2t},
  \qquad y_2(t)=e^{-t}.
\]
The parameter is set to $\varepsilon=10^{-6}$ and the integration interval is
$[0,1]$.  All methods use analytical Jacobians and full Newton iteration.

\begin{table}[tbp]
\centering
\small
\begin{tabular}{clrr}
\toprule
$N$ & method & error at $T=1$ & total Newton iterations\\
\midrule
4 & Proposed & $2.706\times10^{-6}$ & 14\\
4 & Gauss & $4.481\times10^{-3}$ & 12\\
4 & Radau IIA & $7.505\times10^{-5}$ & 12\\
\addlinespace
8 & Proposed & $1.778\times10^{-7}$ & 29\\
8 & Gauss & $1.124\times10^{-3}$ & 24\\
8 & Radau IIA & $9.664\times10^{-6}$ & 24\\
\addlinespace
16 & Proposed & $1.140\times10^{-8}$ & 58\\
16 & Gauss & $2.808\times10^{-4}$ & 48\\
16 & Radau IIA & $1.227\times10^{-6}$ & 48\\
\addlinespace
32 & Proposed & $7.215\times10^{-10}$ & 105\\
32 & Gauss & $6.980\times10^{-5}$ & 96\\
32 & Radau IIA & $1.546\times10^{-7}$ & 96\\
\addlinespace
\bottomrule
\end{tabular}
\caption{Kaps problem with $\varepsilon=10^{-6}$ and $s=2$ active stages. Methods use the same number of uniform steps.}
\label{tab:kaps-comparison-s2}
\end{table}
\begin{table}[tbp]
\centering
\small
\begin{tabular}{clrr}
\toprule
$N$ & method & error at $T=1$ & total Newton iterations\\
\midrule
4 & Proposed & $1.108\times10^{-9}$ & 17\\
4 & Gauss & $4.790\times10^{-5}$ & 12\\
4 & Radau IIA & $4.794\times10^{-8}$ & 12\\
\addlinespace
8 & Proposed & $1.791\times10^{-11}$ & 32\\
8 & Gauss & $3.052\times10^{-6}$ & 24\\
8 & Radau IIA & $1.527\times10^{-9}$ & 24\\
\addlinespace
16 & Proposed & $2.278\times10^{-13}$ & 64\\
16 & Gauss & $1.907\times10^{-7}$ & 48\\
16 & Radau IIA & $4.822\times10^{-11}$ & 48\\
\addlinespace
32 & Proposed & $3.308\times10^{-14}$ & 125\\
32 & Gauss & $1.168\times10^{-8}$ & 96\\
32 & Radau IIA & $1.654\times10^{-12}$ & 96\\
\addlinespace
\bottomrule
\end{tabular}
\caption{Kaps problem with $\varepsilon=10^{-6}$ and $s=3$ active stages. Methods use the same number of uniform steps.}
\label{tab:kaps-comparison-s3}
\end{table}

At equal active-stage count and equal step count, the proposed method is
consistently more accurate in the displayed regime.  For $s=2$ and $N=16$,
its error is about $1.1\times10^{-8}$, compared with
$2.8\times10^{-4}$ for Gauss and $1.2\times10^{-6}$ for Radau IIA.  The total
Newton counts are comparable.  For $s=3$ and $N=8$, the errors are about
$1.8\times10^{-11}$, $3.1\times10^{-6}$, and $1.5\times10^{-9}$,
respectively.

\begin{figure}[tbp]
  \centering
  \begin{subfigure}[t]{0.49\textwidth}
    \centering
    \includegraphics[width=\linewidth]{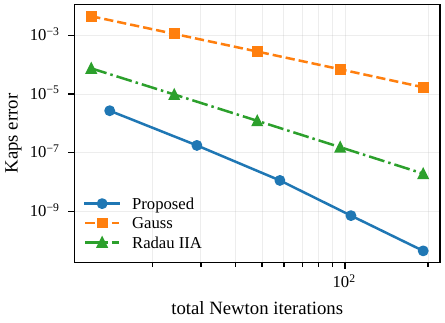}
    \caption{$s=2$.}
    \label{fig:kaps-s2}
  \end{subfigure}\hfill
  \begin{subfigure}[t]{0.49\textwidth}
    \centering
    \includegraphics[width=\linewidth]{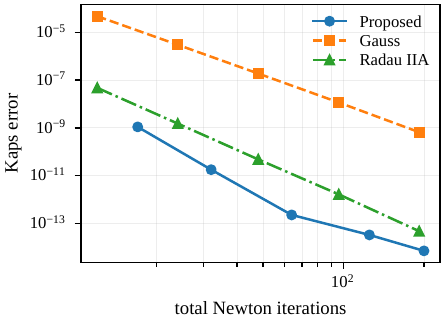}
    \caption{$s=3$.}
    \label{fig:kaps-s3}
  \end{subfigure}
  \caption{Error versus total Newton iterations for the Kaps problem.}
  \label{fig:kaps-work-precision}
\end{figure}

\subsection{A semidiscrete heat equation}

Finally, consider the heat equation
\[
  u_t=u_{xx},
  \qquad 0<x<1,
  \qquad u(t,0)=u(t,1)=0,
\]
discretized by centered second differences at 40 interior points.  The
initial grid vector is
\[
  u_i(0)=\sin(\pi x_i)+0.1\sin(40\pi x_i),
  \qquad x_i=\frac{i}{41}.
\]
The first term is a slowly decaying smooth mode, while the second is a highly
damped grid-scale mode.  The exact semidiscrete solution at $T=0.1$ is
computed by diagonalizing the symmetric discrete Laplacian.

\begin{table}[tbp]
\centering
\small
\begin{tabular}{clrrr}
\toprule
$N$ & method & error & ratio vs. Gauss & ratio vs. Radau\\
\midrule
1 & Proposed & $5.073\times10^{-4}$ & 0.00514 & 0.116\\
1 & Gauss & $9.867\times10^{-2}$ & 1 & 22.6\\
1 & Radau IIA & $4.370\times10^{-3}$ & 0.0443 & 1\\
\addlinespace
2 & Proposed & $3.736\times10^{-5}$ & 0.000401 & 0.0678\\
2 & Gauss & $9.306\times10^{-2}$ & 1 & 169\\
2 & Radau IIA & $5.508\times10^{-4}$ & 0.00592 & 1\\
\addlinespace
4 & Proposed & $2.564\times10^{-6}$ & 3.41e-05 & 0.0356\\
4 & Gauss & $7.508\times10^{-2}$ & 1 & 1.04e+03\\
4 & Radau IIA & $7.204\times10^{-5}$ & 0.00096 & 1\\
\addlinespace
8 & Proposed & $1.683\times10^{-7}$ & 5.29e-06 & 0.0181\\
8 & Gauss & $3.184\times10^{-2}$ & 1 & 3.43e+03\\
8 & Radau IIA & $9.273\times10^{-6}$ & 0.000291 & 1\\
\addlinespace
16 & Proposed & $1.079\times10^{-8}$ & 1.05e-05 & 0.00916\\
16 & Gauss & $1.029\times10^{-3}$ & 1 & 874\\
16 & Radau IIA & $1.177\times10^{-6}$ & 0.00114 & 1\\
\addlinespace
\bottomrule
\end{tabular}
\caption{Semidiscrete heat equation with 40 interior unknowns and $s=2$ active stages. Ratios below one favor the listed method.}
\label{tab:heat-comparison-s2}
\end{table}
\begin{table}[tbp]
\centering
\small
\begin{tabular}{clrrr}
\toprule
$N$ & method & error & ratio vs. Gauss & ratio vs. Radau\\
\midrule
1 & Proposed & $6.056\times10^{-6}$ & 6.28e-05 & 0.0127\\
1 & Gauss & $9.642\times10^{-2}$ & 1 & 202\\
1 & Radau IIA & $4.775\times10^{-4}$ & 0.00495 & 1\\
\addlinespace
2 & Proposed & $6.216\times10^{-8}$ & 7.18e-07 & 0.00723\\
2 & Gauss & $8.661\times10^{-2}$ & 1 & 1.01e+04\\
2 & Radau IIA & $8.595\times10^{-6}$ & 9.92e-05 & 1\\
\addlinespace
4 & Proposed & $1.020\times10^{-9}$ & 1.81e-08 & 0.0198\\
4 & Gauss & $5.640\times10^{-2}$ & 1 & 1.09e+06\\
4 & Radau IIA & $5.156\times10^{-8}$ & 9.14e-07 & 1\\
\addlinespace
8 & Proposed & $1.650\times10^{-11}$ & 1.63e-09 & 0.0116\\
8 & Gauss & $1.015\times10^{-2}$ & 1 & 7.12e+06\\
8 & Radau IIA & $1.426\times10^{-9}$ & 1.41e-07 & 1\\
\addlinespace
16 & Proposed & $2.621\times10^{-13}$ & 2.45e-08 & 0.00582\\
16 & Gauss & $1.072\times10^{-5}$ & 1 & 2.38e+05\\
16 & Radau IIA & $4.501\times10^{-11}$ & 4.2e-06 & 1\\
\addlinespace
\bottomrule
\end{tabular}
\caption{Semidiscrete heat equation with 40 interior unknowns and $s=3$ active stages. Ratios below one favor the listed method.}
\label{tab:heat-comparison-s3}
\end{table}

This example displays the intended combination particularly clearly.  At
coarse and moderate step sizes, Gauss retains too much of the high-frequency
mode, while Radau IIA damps it but carries one order less at the same stage
count.  The proposed method damps the grid-scale component and preserves the
$2s$-th-order slow-mode accuracy.  For $s=3$ and $N=4$, its error is about
$1.0\times10^{-9}$, versus $5.6\times10^{-2}$ for Gauss and
$5.2\times10^{-8}$ for Radau IIA.

\begin{figure}[tbp]
  \centering
  \begin{subfigure}[t]{0.49\textwidth}
    \centering
    \includegraphics[width=\linewidth]{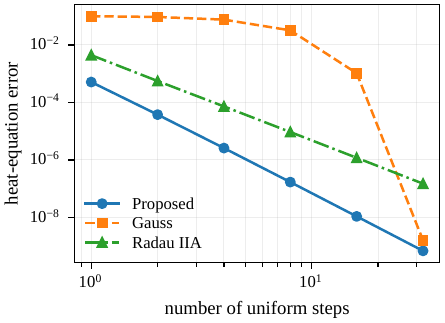}
    \caption{$s=2$.}
    \label{fig:heat-s2}
  \end{subfigure}\hfill
  \begin{subfigure}[t]{0.49\textwidth}
    \centering
    \includegraphics[width=\linewidth]{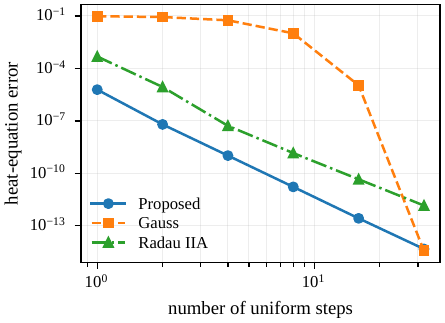}
    \caption{$s=3$.}
    \label{fig:heat-s3}
  \end{subfigure}
  \caption{Fixed-stage comparisons for the semidiscrete heat equation.}
  \label{fig:heat-comparisons}
\end{figure}

\subsection{Interpretation of the comparisons}

The experiments illustrate the principal contribution of the family.
At a fixed active-stage count, Gauss supplies order $2s$ without stiff decay,
whereas Radau IIA supplies stiff decay with order $2s-1$.  The proposed
anchored two-derivative construction supplies both order $2s$ and endpoint
$L$-stability.  On the displayed problems this produces a broad interval of
step sizes in which its endpoint error is lower than that of both classical
comparators.

The advantage is particularly relevant when total-derivative information is
available through Jacobian--vector products, automatic differentiation, or
Lax--Wendroff-type formulations.  In these settings, the construction offers
a compact route to order $2s$ and $L$-stability with only $s$ unknown stage
states.

\section{Discussion and practical considerations}
\label{sec:discussion}

The principal contribution is the simultaneous realization of three
properties for every positive integer $s$: exactly $s$ active unknown stages,
global order $2s$, and the second-subdiagonal Pad\'{e} endpoint stability
function.  The value and total-derivative information at each node supplies
two Hermite data channels while the nonlinear unknown remains one state per
active node; this is the sense in which the formulation is compact.  The
determinant identity additionally gives a transparent common
denominator for the coupled stage system and excludes hidden poles.  Because
the endpoint stability function is independent of the node locations, the
nodes may be chosen for conditioning, physical stage placement, or nonlinear
solution efficiency rather than for scalar stability optimization.

The coupled stage equations have a natural block structure that is compatible
with Newton--Krylov iteration, Jacobian reuse, and block preconditioning.
For high-order coefficient generation, exact or arbitrary-precision arithmetic
and scaled polynomial bases provide reliable alternatives to an unscaled
monomial solve.  The node freedom established by the theory can also be used
to improve conditioning and nonlinear-solver behavior without changing the
stability function.

As usual, $L$-stability refers to the accepted one-step stability function.
The internal transfer functions remain available explicitly as
$G_j(z)=\Pi_j(z)/Q_s(z)$, providing a direct basis for future optimization of
stage amplification, conditioning, and stiffness-uniform performance.

\section{Conclusions}
\label{sec:conclusions}

An arbitrary-order family of implicit anchored two-derivative methods has been
constructed.  The known value $Y_0=u^n$ acts as a Hermite anchor and the method
solves for exactly $s$ active stage states.  Imposing $2s$ Hermite moments gives
global order $2s$ and leaves two coefficients per row.  A Pad\'{e}--Hermite
basis theorem determines these coefficients uniquely for every ordered real
node set.

The central identities are
\[
  \det(I-zA-z^2\widehat A)=Q_s(z),
  \qquad
  R_s(z)=\frac{P_s(z)}{Q_s(z)}=[s-1/s+1]_{e^z}.
\]
They show that the stage realization has no hidden poles and that the accepted
one-step map is $L$-stable (and therefore $A$-stable) for every positive
integer $s$.  Thus the family combines
\[
  \boxed{
  s\text{ active stages},\qquad
  \text{order }2s,\qquad
  L\text{-stability for every }s.
  }
\]

Exact symbolic checks through $s=6$ verify the complete algebraic
construction.  High-precision nonlinear tests confirm orders $2,4,6,$ and $8$
for $s=1,2,3,4$, while endpoint stability diagnostics reproduce the
imaginary-axis bound and stiff decay predicted by the Pad\'{e} theorem.  The
result is a constructive and node-flexible arbitrary-order family whose main
accuracy and stability properties are available for every member of the
sequence.

The explicit common-denominator structure and freedom in the node locations
open several directions for further development, including optimized block
solvers, high-order node design, stage-amplification control, adaptive time
stepping, and stiffness-uniform analysis for semidiscrete partial differential
equations.

\section*{Acknowledgements}
This work was supported by the Natural Science Foundation of the Department
of Education of Henan Province (Grant No.~26A110007), the Department of
Science and Technology of Henan Province (Grant No.~252300423500), and the
Doctoral Foundation of Henan Polytechnic University (Grant No.~B2024-60).
The author is the principal investigator of all three projects and thanks the
High-Performance Computing Center of Henan Polytechnic University for
providing computational resources.

\end{document}